\documentclass[letterpaper, 12pt]{article}

\usepackage[margin=1in]{geometry}
\usepackage{amsfonts, amsmath, amsthm, bm}
\usepackage{authblk}
\usepackage{url}
\usepackage{hyperref}
\hypersetup{
}
\usepackage{color}
\usepackage[ruled,vlined]{algorithm2e}

\usepackage{graphicx}
\graphicspath{ {./images/} }

\usepackage{mathtools}
\newcommand{\defeq}{\vcentcolon=}
\newcommand{\eqdef}{=\mathrel{\mathop:}}

\newcommand{\RR}{\mathbb{R}}
\newcommand{\EE}{\mathbb{E}}
\newcommand{\VV}{\mathbb{V}}
\newcommand{\NN}{\mathbb{N}}
\newcommand{\dd}{\mathrm{d}}

\newtheorem{theorem}{Theorem}[section]

\newtheorem{lemma}[theorem]{Lemma}
\newtheorem{proposition}[theorem]{Proposition}
\newtheorem{definition}[theorem]{Definition}
\newtheorem{assumption}[theorem]{Assumption}

\title{Stochastic Gradient Descent in the Viewpoint of Graduated Optimization}
\author[1]{Da Li}
\author[1]{Jingjing Wu}
\author[1]{Qingrun Zhang}
\affil[1]{Department of Mathematics and Statistics, University of Calgary}

\date{}

\begin{document}

\maketitle

\begin{abstract}

    Stochastic gradient descent (SGD) method is popular for solving non-convex optimization problems in machine learning.
    This work investigates SGD from a viewpoint of graduated optimization, which is a widely applied approach for non-convex optimization problems.
    Instead of the actual optimization problem, a series of smoothed optimization problems that can be achieved in various ways are solved in the graduated optimization approach.
    In this work, a formal formulation of the graduated optimization is provided based on the nonnegative approximate identity, which generalizes the idea of Gaussian smoothing.
    Also, an asymptotic convergence result is achieved with the techniques in variational analysis.
    Then, we show that the traditional SGD method can be applied to solve the smoothed optimization problem.
    The Monte Carlo integration is used to achieve the gradient in the smoothed problem, which may be consistent with distributed computing schemes in real-life applications.
    From the assumptions on the actual optimization problem, the convergence results of SGD for the smoothed problem can be derived straightforwardly.
    Numerical examples show evidence that the graduated optimization approach may provide more accurate training results in certain cases. 

\end{abstract}

\section{Introduction} \label{sec:introduction}

Consider the non-convex optimization problem:
$$ \min_x f(x), \quad x \in \Omega,$$
where $f$ is the objective function which is non-convex and smooth enough, $\Omega$ is a feasible set.
Many machine learning problems and deep neural network training problems can be summarized as the above non-convex optimization problem. 
The stochastic gradient descent (SGD) method and its variants, such as Adagrad \cite{duchi2011adaptive} and Adam \cite{kingma2014adam}, are among the most important tools in machine learning.
The iterative updates in SGD method performs in a way as
$$x_{k+1} = x_k - \alpha_k g(x_k, \xi_k),$$
where $\alpha_k$ is the step size and $g$ is the stochastic gradient which satisfies some additional assumptions on its expectation and variations.
In the classic SGD setting, the noise of the stochastic gradient comes from the dataset $\xi_k$.
In this work, we consider the case where the noise is not only caused by the data, but also caused by the model $x_k$ in the current iteration.
This idea leads to the graduated optimization approach.

The graduated optimization (or named continuation method) is a popular heuristic approach for solving non-convex optimization problems.
A sequence of subproblems with different smoothed objective functions is defined before the optimization process begins.
Starting with the smoothest subproblem, the sequence of subproblems is solved sequentially, where the solution of each subproblem serves as the initial value for the following subproblem.
Heuristically, this approach may lead to the global minimum solution for the non-convex optimization problem instead of a local minimum solution.

\begin{figure}
    \centering
    \includegraphics[width=0.5\textwidth]{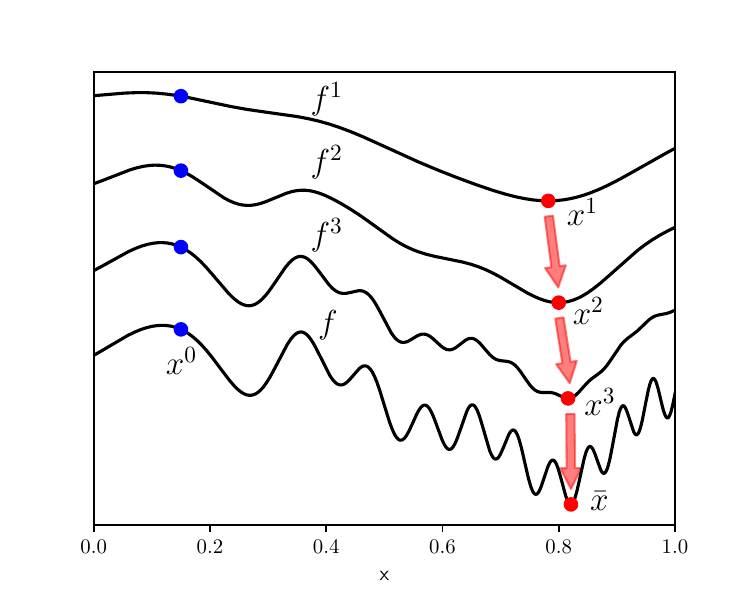}
    \caption{A demonstration of graduated optimization approach.}
    \label{fig:demo_GO}
\end{figure}

\textbf{Motivation.}
The motivation of the graduated optimization approach can be demonstrated in Figure \ref{fig:demo_GO}.
If the objective function $f$ is minimized with the initial point $x_0$, the local minimum around $x=0.25$ will be achieved as the solution.
With the graduated optimization approach, the smoothed objective function $f_1$ is minimized with the initial value $x_0$ at the first stage, and its solution is $x_1$.
Next, the smoothed objective function $f_2$ is minimized with point $x_1$ as the initial value.
Then, solution $x_2$ is used as the initial value of minimizing $f_3$, and so on.
By this construction, we can assume $x_k \to x^*$ which is the global minimum of $f$ as $k \to \infty$.
Notice that if we start with a subproblem in which the objective function is not smooth enough, the solution may still be a local minimum.
For example, when we minimize $f_2$ with $x_0$ as the initial value, the solution will still be a local minimum near $x = 0.25$.

\textbf{Previous work.}
The works on the graduated optimization approach date back to the 1980s under the name continuation method \cite{witkin1987signal}, graduated non-convex (GNC) algorithm \cite{blake1987visual}, and mean field annealing \cite{yuille1989energy}.
Later, a theoretical analysis result is provided in \cite{wu1996effective} from an optimization view:
under certain second order regularity conditions, for any stationary point of the actual objective function, there exists a continuous curve consisting of the stationary point of the objective function of the subproblems.

Since it was proposed, the concept of graduated optimization was successfully applied in many computer vision problem, in an explicitly or implicitly way \cite{zerubia1993mean,nikolova2010fast,mobahi2012seeing}.
Also, the graduated optimization approach has been applied to many application scenarios in the machine learning field, such as semi-supervised learning \cite{chapelle2006continuation}, unsupervised learning \cite{smith2004annealing}, and ranking \cite{chapelle2010gradient}.
In \cite{bengio2009learning}, the author suggested some form of continuation in learning has played an important role in the recent developments in the training of deep architectures \cite{hinton2006fast,erhan2009difficulty}.
A comprehensive survey on the development and application of the graduated optimization approach is provided in \cite{mobahi2015link}.

While it was popular in the machine learning field, theoretical analysis of the graduated optimization was still under development.
A bound on the endpoint solution of the continuation method with the Gaussian smoothing is provided in \cite{mobahi2015theoretical}, with no practical algorithms provided.
Later, a special kind of non-convex objective function, $(a,\sigma)$-nice function, is studied in \cite{hazan2016graduated}.
An implicit smoothing is provided by sampling oracles, and a graduated optimization with gradient oracle algorithm is provided.
The convergence to the global optimum for the $(a,\sigma)$-nice function is provided, and the convergence speed of the algorithm is provided.
The smoothing effect of the graduated optimization approach is not only studied as an optimization strategy but also used to illustrate why the SGD method performs well in practical machine learning tasks \cite{kleinberg2018alternative}.
Based on this viewpoint, a perturbed SGD algorithm is proposed in \cite{Harshvardhan2021Escaping}, where a global convergence result is provided based on the P\L \ condition.
Furthermore, an entropy-SGD method is proposed with a similar smoothing explanation \cite{chaudhari2019entropy}, where the algorithm is a composition of two nested SGD loops, and Langevin dynamics is used to compute the gradient of the local entropy.

\textbf{Contribution.}
In this work, we provide a formulation of the graduated optimization approach based on the nonnegative approximate identity, in which Gaussian smoothing is a special case.
Then, we provide an asymptotic convergence result of the graduated optimization approach with the techniques in variational analysis and no further requirement of special kinds of functions.
We verified that for the smoothed subproblems of the graduated optimization, when its gradient is obtained by the Monte Carlo integration method, the assumptions about the original problem can be inherited into the subproblems.
In this way, many of the convergence results of solving subproblems with the traditional SGD method can be obtained straightforwardly from the convergence result of the actual optimization problem.
Based on this, a multi-layer SGD scheme is formally proposed.
Synthetic numerical examples show the evidence that the graduated optimization approach provides more accurate optimization results in low-dimensional settings.

\textbf{Content.}
The content of this paper is organized as follows:
In section 2, we discuss a formulation of the graduated optimization approach. The asymptotic convergence result is discussed.
The multi-layer SGD scheme is investigated in section 3.
Section 3 shows that the smoothed subproblem in the graduated optimization approach can be solved with the SGD method and the Monte Carlo random sampling technique.
The multi-layer SGD (ML-SGD) is formally provided based on this method.
We show that similar convergence results can be achieved for the subproblem with the assumption of the actual optimization problem.
Two synthetic numerical examples are investigated in section 4.
Additional discussion and outlook are provided in section 5.

\section{Formulation of the graduated optimization approach}

Rewrite the non-convex optimization problem:
\begin{align} \label{prob:opt_prob}
    (\mathrm{P}) \quad \min_x f(x), \quad \text{such that } x \in \Omega,
\end{align}
where $f \in L^1(\RR^n)$.
Denote the actual optimization problem as $(\mathrm{P})$.
The constraint set $\Omega$ is nonempty, closed, and bounded.
In practice, we can assume that the $\Omega$ as a simple box constraint: 
\begin{align}
    \Omega = \{ x \in \RR^n \ \vert \ a \leq x^i \leq b\},
\end{align}
where $a,b \in \RR$ and $a \leq b$.
The constraint set $\Omega$ introduced in the above problem is for the convenience of analysis.
In this work, we assume that the $\Omega$ is large enough such that the optimization process of the above problem is in $\Omega$ and never reaches the boundary of $\Omega$.
Notice that, we are not assuming $\Omega$ is convex. When $\Omega$ is convex, other algorithms for the constrained optimization problem can be developed, and we will not discuss it in this work for simplicity.

\begin{definition}[Approximate identity]
    An approximate identity is a family of functions $\phi_t \in L^1(\RR^n)$ for $t > 0$ such that\\
    (a) $\int_{\RR^n} |\phi_t(x)| \ \dd x \leq A$, where $A$ is a constant independent of $t$.\\
    (b) $\int_{\RR^n} \phi_t(x) \ \dd x = 1$, for all $t>0$. \\
    (c) $\lim_{t\to 0} \int_{|x| \geq \delta} |\phi_t(x)| \ \dd x = 0$, for any $\delta >0$.
\end{definition}
Next, generate an approximate identity with a kernel function $\phi$.
Given $\phi \in L^1(\RR^n)$ with 
\begin{align}
    \int_{\RR^n} \phi(x) \ \dd x = 1.
\end{align}
Then, for $t > 0$, let
\begin{align}
    \phi_t(x) = t^{-n} \phi (t^{-1}x).
\end{align}
It can be easily checked that $\{\phi_t\}$ is an approximate identity.
It is well known that the Gaussian kernel can be formulated as an approximate identity as
\begin{align}\label{eq:heat_kernel}
    \phi_t(x) = (4\pi t)^{-n/2} e^{-|x|^2 / 4 t}.
\end{align}
Given a series nonnegative approximate identities $\{\phi_t\}$, construction a series of smoothed objective functions as
\begin{align} \label{eq:construct_tilde_f}
    \tilde f^t = f * \phi_t,
\end{align}
where $*$ is the convolution operator.
Let $\nu = \lceil 1/t \rceil$, here $\nu$ is a smooth coefficient. Then, define function $f^\nu$ as
\begin{align}
    f^\nu = \tilde f^{\lceil 1/t \rceil},
\end{align}
where $\nu \in \NN$.
Notice that $\{f^\nu\}$ is a subsequence of $\{\tilde f^t\}$.
It can be shown that $f^\nu \to f$ uniformly as $\nu \to \infty$ by Proposition \ref{prop:8.14}.
For each $f_\nu$, define the $\nu$th smoothed subproblem of $(\mathrm{P})$ as
\begin{align}
    (\mathrm{P}^\nu)\quad \min_x f^\nu(x), \quad \text{such that } x \in \Omega.
\end{align}
By this method, a series of smoothed subproblems $(P^\nu)$ of the actual optimization problem is constructed.
Then, a formal algorithm for the graduated optimization can be achieved.

\begin{algorithm}
    \caption{A formal algorithm for the graduated optimization}
    \KwIn{objective function $f$; smooth kernel $\phi$; $\nu_0 < \nu_1 < \cdots < \nu_{N_m} < \infty$\;}
    Find the global minimum of $(P^{\nu_0})$, denoted as $x^0$\;
    \For{$m = 1, 2, \cdots, N_m$}{
        With the initial value $x^{m-1}$, solve the subproblem $(P^{\nu_m})$. Denote the solution as $x^{m}$.}
    \KwOut{$x^{N_m}$}
\end{algorithm}

Next, we discuss the asymptotic convergence of the graduated optimization approach.
\begin{definition}[$\varepsilon$-optimality]
    The $\varepsilon$-optimal solution of minimizing a proper function $f$ on $\RR^n$ can be denoted as a set
    \begin{align}
        \varepsilon \textrm{-} \arg\min f = \{x | f(x) \leq \inf f + \varepsilon\}.
    \end{align}
\end{definition}

Denote the set of the global minimum of the actual problem $(\mathrm{P})$ as $S$, and the set of the global minimum of $\nu$th subproblem $(\mathrm{P}^\nu)$ as $S^\nu$, i.e.,
\begin{align}
    S \defeq \arg\min_{x\in\Omega} f(x), \quad S^\nu \defeq \arg\min_{x\in\Omega} f^\nu(x).
\end{align}
Also, the $\varepsilon$-optimal solution of function $f$ and $f^\nu$ can be denoted as
\begin{align}
    \varepsilon \textrm{-} S \defeq \varepsilon \textrm{-} \arg\min f, \quad \varepsilon \textrm{-} S^\nu \defeq \varepsilon \textrm{-} \arg\min f^\nu.
\end{align}
Our main theorem states the relation between $S^\nu$ and $S$ as $\nu \to \infty$.

\begin{theorem} [Main Theorem] \label{thm:main_result1}
    For the actual optimization problem $(\mathrm{P})$, construct a series of smoothed subproblem $\{ (\mathrm{P}^\nu) \}$.
    There exists an index $N$ such that when $\nu \geq N$, the sets $S^\nu $ are nonempty and form a bounded sequence with
    \begin{align}
        \limsup_\nu S^\nu \subset S.
    \end{align}
    For any $\varepsilon^\nu \to 0$ and $x^\nu \in \varepsilon^\nu \textrm{-} S^\nu $, the sequence $\{x^\nu\}$ is bounded and such that all its cluster points belong to $S$.
    If $S$ consists of a unique point $\bar x$, then $x^\nu \to \bar x$.
\end{theorem}
The proof is provided in the following subsections.

Theorem \ref{thm:main_result1} provides an asymptotic convergence result of the graduated optimization approach.
For large classes of functions ($L^1(\RR^n)$), the convergence of the graduated optimization approach is guaranteed by the above theorem compared to the previous results in \cite{wu1996effective,hazan2016graduated}. 
Moreover, the kernel functions that construct the smoothed objective functions are not limited to the Gaussian kernel.
However, the above theorem does not provide an accurate convergence rate as the work \cite{hazan2016graduated}.
How to choose the smooth kernel $\phi$ and the smooth coefficients $\nu$ is still a heuristic decision.

\subsection{Prerequisite}

In this subsection, we briefly review the prerequisite in variational analysis.
Most of the contents can be found in the textbook \cite{rockafellar2009variational} and \cite{folland1999real}.

Denote a subset of $\NN$ to represent the index of convergence sequences as
\begin{align}
    \mathcal{N}_\infty \defeq \{N \subset \NN \ \vert \ \NN / N \text{ finite}\}.
\end{align}
This notation is useful when we representing the index $\nu \to \infty$ for $\nu \in \NN$.

\begin{definition}[Eventually level boundedness]
    A sequence of sets $C^\nu \subset \RR^n$ is eventually bounded that for some index set $N \in \mathcal{N}_\infty$ if the set $\cup_{\nu \in N} C^\nu$ is bounded. \\
    A sequence of functions $f^\nu$ is eventually level-bounded if for each $\alpha \in \RR^n$, the sequence of sets $\mathrm{lev}_{\leq \alpha} f^\nu$ is eventually bounded.
\end{definition}

\begin{proposition}[\cite{rockafellar2009variational} 7.32 (a)] \label{prop:eventually_boundedness}
    The sequence $\{f^\nu \}$ is eventually level-bounded if the sequence of sets $\text{dom } f^\nu$ is eventually bounded.
\end{proposition}

\begin{proposition}[\cite{rockafellar2009variational} 7.2] \label{prop:epi_conv}
    Let ${f^\nu}$ be any sequence of functions on $\RR^n$, and let $x$ be any point of $\RR^n$. Then, $f^\nu \overset{e} \to f$ if and only if at each point $x$ one has
    \begin{align}
        \lim\inf_\nu f^\nu (x^\nu) \geq f(x) \quad \text{for every sequence } x^\nu \to x, \label{eq:epi_conv_liminf}\\
        \lim\sup_\nu f^\nu (x^\nu) \leq f(x) \quad \text{for some sequence } x^\nu \to x. \label{eq:epi_conv_limsup}
    \end{align}
\end{proposition}

\begin{proposition}[\cite{rockafellar2009variational} 1.6] \label{prop:lsc_property}
    For a function $f:\RR^n \to \bar \RR$, $f$ is lower semicontinuous on $\RR^n$ if and only if the level sets of type $\text{lev}_{\leq \alpha} f$ are closed in $\RR^n$.
\end{proposition}

The following proposition provides the uniform convergence of $f * \phi_t$.
\begin{proposition}[\cite{folland1999real} Theorem 8.14]\label{prop:8.14}
    Suppose $\phi \in L^1(\RR^n)$ and $\int \phi(x) \ \dd x = 1$.
    If $f$ is bounded and uniformly continuous, then $f * \phi_t \to 1 f$ uniformly as $t \to 0$.
\end{proposition}

The following proposition shows the regularity of the objective functions in the subproblem.
\begin{proposition}[\cite{folland1999real}, Proposition 8.10]\label{prop:8.10}
    If $f \in L^1$, $g \in C^k$, and $\partial^\alpha g$ is bounded for $|\alpha| \leq k$, then $f * g \in C^k$ and $\partial^\alpha(f * g) = f * (\partial^\alpha g)$ for $|\alpha| \leq k$.
\end{proposition}

The following proposition provides the connection between the uniform convergence and epigraph convergence for the sequence $f^{\nu}$.
\begin{proposition}[\cite{rockafellar2009variational} 7.15] \label{prop:7.15}
    Consider $f^\nu, f: \RR^n \to \bar \RR$ and a set $X \subset \RR^n$.
    If the functions $f^\nu$ are lower semicontinuous (lsc) relative to $X$ and converge uniformly to $f$ on $X$, then $f$ is lsc relative to $X$.
    Also, $f^\nu$ epi-converge to $f$ relative to $X$, written as $f^\nu \overset{e} \to f$ relative to $X$.
\end{proposition}

\begin{proposition}[\cite{rockafellar2009variational} 7.33 convergence in minimization] \label{prop:7.33}
    Suppose the sequence $\{f_\nu\}$ is eventually level-bounded, and  $f^\nu \overset{e} \to f$ with $f^\nu$ and $f$ lsc and proper. Then
    \begin{align*}
        \inf f^\nu \to \inf f \quad (\text{finite}),
    \end{align*}
    while for $\nu$ in some index set $N \in \mathcal{N}_\infty$ the sets $\arg\min f^\nu$ are nonempty and form a bounded sequence with
    \begin{align*}
        \lim\sup_\nu (\arg\min f^\nu) \subset \arg\min f.
    \end{align*}
    Indeed, for any choice of $\varepsilon^\nu \to 0$ and $x^\nu \in \varepsilon^\nu \textrm{-} \arg\min f^\nu$, the sequence $\{x^\nu\}_{\nu\in\NN}$ is bounded and such that all its cluster points belong to $\arg\min f$.
    If $\arg\min f$ consists of a unique point $\bar x$, one must actually have $x^\nu\to\bar x$.
\end{proposition}

\subsection{Proof of the main theorem}
The following lemma provides the epi-convergence result of the objective functions of the subproblems $(\mathrm{P}^\nu)$.
\begin{lemma} \label{lemma:epi_conv_fnu}
    The objective function $f^\nu$ of the subproblem $(\mathrm{P}_\nu)$ epi-converge to the objective function $f$ of $(\mathrm{P})$ relative to $\Omega$, i.e., $f^\nu \overset{e} \to f$ relative to $\Omega$.
\end{lemma}
\begin{proof}
    Since ${\phi_t}$ is an nonnegative approximate identity, by the construction of sequence $\{\tilde f^t\}$ (equation \eqref{eq:construct_tilde_f}), $\tilde f^t \to f$ uniformly on $\Omega$ as $t \to 0$ by Proposition \ref{prop:8.14}.
    Since $f$ is continuous and $\phi_t \in L^1(\RR^n)$, $\tilde f^t$ is continuous relative to $\Omega$ for all $t > 0$ by Proposition \ref{prop:8.10}.
    Since ${f^\nu}$ is a subsequence of $\{\tilde f^t\}$, the proof is finished by Proposition \ref{prop:7.15}.
\end{proof}

Proof of the main theorem.
\begin{proof}
    \textbf{1. Construct extended real-valued functions.}

    For the objective functions of problem $(\mathrm{P})$ and $(\mathrm{P}^\nu)$, define extended real-valued functions:
    \begin{align} \label{eq:construction_g}
        g(x) &= \begin{cases}
            f(x), \quad &x \in \Omega, \\ \infty, \quad &x \notin \Omega,
        \end{cases}\\
        g^\nu(x) &= \begin{cases}
            f^\nu(x), \quad &x \in \Omega, \\ \infty, \quad &x \notin \Omega,
        \end{cases}
    \end{align}
    Obviously, problems $(\mathrm{P})$ and $(\mathrm{P}^\nu)$ are respectively equivalent to
    \begin{align}
        (\mathrm{\tilde P}) \quad \min_{x\in\RR^n} g(x), \quad \text{and} \quad (\mathrm{\tilde P^\nu}) \quad \min_{x\in\RR^n} g^\nu(x).
    \end{align}
    Also,
    \begin{align} 
        \arg\min_{x\in\Omega} f(x) &= \arg\min_{x\in\RR^n} g(x), \label{eq:argminf}\\
        \arg\min_{x\in\Omega} f^\nu(x) &= \arg\min_{x\in\RR^n} g^\nu(x). \label{eq:argminfnu}
    \end{align}

    \noindent \textbf{2. Eventually level boundedness of $\{g^\nu\}$.}

    By the construction of $g$ and $g^\nu$, $\text{dom } g \subset \Omega$, $\text{dom } g^\nu \subset \Omega$.
    It is easy to see the sequence of sets $\text{dom } g^\nu$ is a eventually bounded.
    By Proposition \ref{prop:eventually_boundedness}, the sequence of functions $\{g^\nu\}$ is eventually level-bounded.

    \noindent \textbf{3. Epi-convergence of $\{g^\nu\}$.}

    By Lemma \ref{lemma:epi_conv_fnu}, we have $f^\nu \overset{e} \to f$ relative to $\Omega$, and equations \eqref{eq:epi_conv_liminf} and \eqref{eq:epi_conv_limsup} hold. 
    Then, we show $g^\nu \overset{e} \to g$ by Proposition \ref{prop:epi_conv}.

    Case 1: $x \notin \Omega$. For any sequence $x^\nu \to x$,
    \begin{align}
        \liminf_\nu g^\nu(x^\nu) = \infty, \quad \limsup_\nu g^\nu(x^\nu) = \infty.
    \end{align}
    Also, $g(x) = \infty$, equations \eqref{eq:epi_conv_liminf} and \eqref{eq:epi_conv_limsup} hold.

    Case 2: $x \in \Omega$ and $x \notin \partial \Omega$. For any sequence $x^\nu \to x$,
    \begin{align}
        \liminf_\nu g^\nu(x^\nu) = \liminf_\nu f^\nu(x^\nu), \quad \limsup_\nu g^\nu(x^\nu) = \limsup_\nu f^\nu(x^\nu).
    \end{align}
    Also, $g(x) = f(x)$. Then, equations \eqref{eq:epi_conv_liminf} and \eqref{eq:epi_conv_limsup} hold.

    Case 3: $x \in \partial \Omega$.
    For any sequence $x^\nu \to x$, consider the case there is a subsequence $x^{\nu_k} \to x$ with $x^{\nu_k} \in \RR^n / \Omega$.
    Then,
    \begin{align}
        \liminf_\nu g^\nu (x^\nu) = \infty \geq g(x).
    \end{align}
    For any sequence $x^\nu \to x$ with no subsequence $x^{\nu_k} \to x$ with $x^{\nu_k} \in \RR^n / \Omega$,
    \begin{align}
        \liminf_\nu g^\nu (x^\nu) = \liminf_\nu f^\nu(x^\nu) = f(x) = g(x).
    \end{align}
    Then, for any sequence $x^\nu \to x$,
    \begin{align}
        \liminf_\nu g^\nu (x^\nu) \geq g(x).
    \end{align}
    On the other hand, there exists a sequence $x^\nu \to x$ with all $x^\nu \in \Omega$.
    In this case,
    \begin{align}
        \limsup_\nu g^\nu (x^\nu) = \limsup_\nu f^\nu (x^\nu) \leq f(x) = g(x).
    \end{align}
    Equations \eqref{eq:epi_conv_liminf} and \eqref{eq:epi_conv_limsup} hold. Then, $g^\nu \overset{e} \to g$.

    \noindent \textbf{4. Continuity of $g$ and $g^\nu$.}

    By the construction of $f^\nu$ and Proposition \ref{prop:8.10}, $f^\nu$ is continuous in $\Omega$.
    Then, by the construction of $g$ and $g^\nu$, $g$ and $g^\nu$ are proper.
    Since $g$ and $g_\nu$ are continuous on a closed set $\Omega$, then the level set $\mathrm{lev}_{\leq \alpha} g$ and $\mathrm{lev}_{\leq \alpha} g_\nu$ are closed for all $\alpha \in \RR$.
    Then $g$ and $g_\nu$ are lsc by Proposition \ref{prop:lsc_property}.

    Notice that by equations \eqref{eq:argminf} and \eqref{eq:argminfnu}, the solution set of $(\mathrm{P})$ and $(\mathrm{P}_\nu)$ are equivalent to the solution set of $(\mathrm{\tilde P})$ and $(\mathrm{\tilde P}^\nu)$.
    Let $S = \arg\min_{x\in\Omega} f(x)$ and $S^\nu = \arg\min_{x\in\Omega} f^\nu(x)$.
    With all the discussions above, the proof is finished by Proposition \ref{prop:7.33}.

\end{proof}

\section{Multi-layer SGD scheme}
This section discusses a multi-layer SGD (ML-SGD) scheme.
Based on the graduated optimization approach discussed in the previous section, the optimization process can be formulated as solving a series of sequential subproblems.
In many machine learning scenarios, the objection function of the actual optimization problem $(\mathrm{P})$ can be written as
\begin{align}
    f(x) = \int_{\Xi} L(x,\xi) \ \dd P(\xi),
\end{align}
where $\xi$ is the random variables with distribution $P$, which is supported on $\Xi \subset \RR^d$.
When we solve the actual optimization problem $(\mathrm{P})$ with the SGD method, at the $k$th iteration, the stochastic gradient is denoted as $g(x_k, \xi_k)$ with some additional assumptions are satisfied.

Recall the discussion in the previous section, a series of smoothed subproblems are required for the graduated optimization approach.
Denote the objective function of the $\nu$th subproblem $(\mathrm{P}^\nu)$ as
\begin{align}
    f^\nu(x) = f * \phi_{1/\nu} (x) = \int_{\RR^d} f(y) \phi_{1/\nu}(x-y) \ \dd y,
\end{align}
where $\phi_{1/\nu}$ is a kernel discussed in the previous section.
Then, the gradient of $f^\nu(x)$ is given by the Bochner integral:
\begin{align}
    \nabla f^\nu (x) = \int_{\RR^d} \nabla f(y) \phi_{1/\nu}(x-y) \ \dd y.
\end{align}
Next, apply the SGD method to solve the subproblem $(\mathrm{P}^\nu)$.
Based on the above equation, by replacing the gradient with the stochastic gradient, denote the stochastic gradient for the smoothed objective function $f^\nu(x)$ at $k$th iteration as
\begin{align} \label{eq:sto_grad_subproblem0}
    g^\nu (x_k, \xi_k) = \int_{\RR^d} g(y,\xi_k) \phi_{1/\nu}(x_k - y)\ \dd y.
\end{align}

Notice that there is an integral in the above stochastic gradient $g^\nu (x, \xi_k)$.
One straightforward way is to evaluate the integral with Monte Carlo random sampling.
Denote
\begin{align} \label{eq:sto_grad_subproblem}
    g_{N_\nu}^\nu (X_k^\nu, \xi_k) = \frac{1}{N_\nu} \sum_{i=1}^{N_\nu} g ((X_k^\nu)_i, \xi_k),
\end{align}
where $\xi_k$ is still the same random variable, $X_k^\nu$ is a random variable with probability density function $\phi_{1/\nu} (x_k - \cdot)$, which is a kernel function.
When the Gaussian kernel \eqref{eq:heat_kernel} is used, $\phi_{1/\nu} (x_k - \cdot)$ is the Gaussian distribution centered at $x_k$.
This is the case of applying Gaussian smoothing.
Figure \ref{fig:grad_demo} provides a demonstration of gradient evaluation by Monte Carlo random sampling.

\begin{figure}[h]
    \centering
    \includegraphics[width=0.5\textwidth]{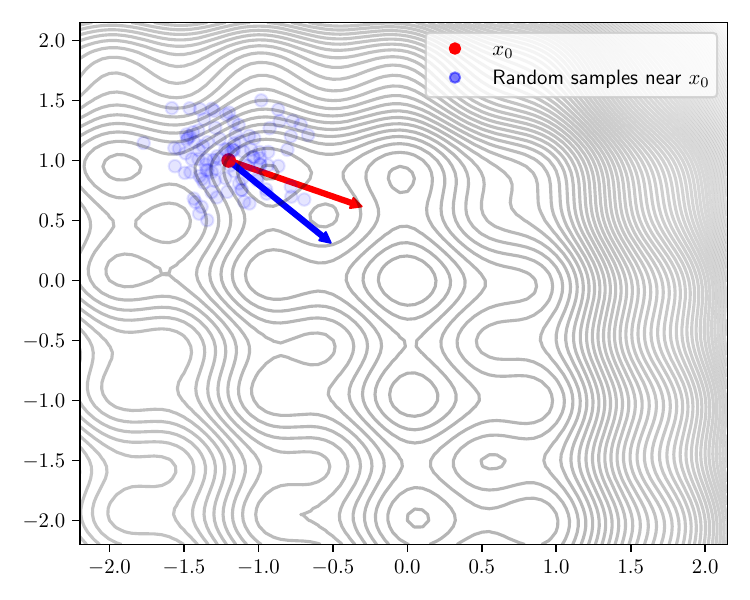}
    \caption{A demonstration of gradient evaluation. The update direction of the actual objective function at $x_0$ is given by the red arrow. The update direction of the smoothed objective function at $x_0$ is given by the blue arrow.}
    \label{fig:grad_demo}
\end{figure}

Algorithm \ref{algorithm:SGD_subproblem} is to apply the SGD method for solving $(\mathrm{P}^\nu)$ with random sampling.
\begin{algorithm}
    \label{algorithm:SGD_subproblem}
    \caption{SGD method solving $(\mathrm{P}^\nu)$ with random sampling}
    Input: $\phi_{1/\nu}$, $x_0^\nu$ \\
    \For(){$k = 1, 2, \cdots$}{
        Generate realizations of random variable $\xi_k$ with $P$, $X_k^\nu$ with $\phi_{1/\nu} (x_k^\nu - \cdot)$.\\
        Compute the stochastic gradient $g_{N_\nu}^\nu(X_k^\nu,\xi_k)$ by equation \eqref{eq:sto_grad_subproblem}. \\
        Update $x_{k+1}^\nu = x_k^\nu - \alpha_k g_{N_\nu}^\nu(X_k^\nu,\xi_k)$, where $\alpha_k$ is the stepsize.
    }
\end{algorithm}

Recall the updates in the classic SGD method can be written as
\begin{align}
    x_{k+1} = x_k - \alpha_k g(x_k, \xi_k),
\end{align}
where $\xi_k$ represents the sample or the batch of samples used in the current iteration.
In this case, the randomness of the stochastic gradient is caused by the data.
However, in Algorithm \ref{algorithm:SGD_subproblem}, the randomness of the stochastic gradient is caused not only by the data but also by the model.
Later we show that when certain assumptions are met, as long as the convergence result for solving the actual optimization problem ($\mathrm{P}$) with the SGD method is available, similar conclusions can also be made for solving the subproblem ($\mathrm{P}^\nu$) with Algorithm \ref{algorithm:SGD_subproblem}.

A formal ML-SGD scheme is provided by Algorithm \ref{algorithm:ML-SGD}.
\begin{algorithm}
    \label{algorithm:ML-SGD}
    \caption{A multi-layer SGD scheme with random sampling}
    Input: an index set $\{\nu_1, \cdots, \nu_{N_m}\}$ where $\nu_{m} \leq \nu_{m-1}$, \\
    \quad the kernel function $\phi$, the initial value $x^0$: \\
    \For(){$m = 1, 2, \cdots, N_m$}{
        Construct the kernel function $\phi_{1/\nu}$.\\
        With the initial value $x^{m-1}$, solve the $m$th subproblem by Algorithm \ref{algorithm:SGD_subproblem}, denote the result as $x^{m}$
    }
    Return: $x^{N_m}$.
\end{algorithm}

Notice that, Algorithm \ref{algorithm:ML-SGD} is a formal scheme that provides no guarantee of converging to the global minimum.
Indeed, Theorem \ref{thm:main_result1} states that the above multi-layer approach can provide a global minimum result.
However, we still need to emphasize that the smooth kernel $\phi$ and smooth coefficients ${\nu}$ are hyperparameters that need to be determined heuristically.

\subsection{An example of convergence analysis}
In this subsection, we show that the convergence result of the subproblem $(\mathrm{P}^\nu)$ can be directly achieved with the assumptions of the actual problem $(\mathrm{P})$.

First, we discuss the Lipschitz smooth gradient assumption:
\begin{assumption} \label{assump:Lipschitz}
    Assume $f: \RR^d \to \RR$ is continuously differentiable.
The gradient of $f$, $\nabla f: \RR^d \to \RR^d$ is Lipschitz continuous with Lipschitz constant $L > 0$, i.e. for all $x, y \in \RR^d$:
\begin{align}
    \|\nabla f(x) - \nabla f(y)\| \leq L\|x-y\|.
\end{align}
\end{assumption}

It is easy to see that the Lipschitz smoothness of $\nabla f$ can be inherited by $\nabla f^\nu$.
\begin{lemma} \label{lemma:Lipschitz_subproblem}
    For the objective function of $\nu$-th subproblem $f^\nu$, the gradient $\nabla f^\nu: \RR^d \to \RR^d$ is Lipschitz continuous with Lipschitz constant $L$.
\end{lemma}
The proof is deferred to the following subsection. 

The following assumptions are made for the stochastic gradient at $k$th iteration $g(x_k, \xi_k)$.
\begin{assumption} \label{assump:stochastic_gradient}
    For any $k \geq 1$, $x\in\Omega$, we make assumptions\\
    (a) $\EE \left[g(x, \xi_k)\right] = \nabla f(x)$,\\
    (b) $\EE \left[\left\|g(x,\xi_k) - \nabla f(x) \right\|^2\right] \leq \sigma^2$.
\end{assumption}

The following lemma states that the stochastic gradient $g^\nu (x, \xi_k)$ is an unbiased estimator of $\nabla f^\nu (x_k)$, and the variance of the random variable $\left\| g^\nu (x_k, \xi_k) - \nabla f^\nu (x_k) \right\|$ is bounded.
\begin{lemma} \label{lemma:stochastic_gradient_1}
    For the stochastic gradient $g^\nu (x, \xi_k)$ in equation \eqref{eq:sto_grad_subproblem0}, by Assumption \ref{assump:stochastic_gradient}, we have\\
    (a) $\EE \left[ g^\nu (x_k, \xi_k)\right] = \nabla f^\nu (x_k)$,\\
    (b) $\EE \left[ \left\| g^\nu (x_k, \xi_k) - \nabla f^\nu (x_k) \right\|^2\right] \leq \sigma^2$.
\end{lemma}
The proof is deferred to the following subsection. 

The following lemma states that the stochastic gradient $g_{N_\nu}^\nu (X_k^\nu, \xi_k)$ is an unbiased estimator of $\nabla f^\nu (x_k)$, and the variance of the random variable $\left\| g_{N_\nu}^\nu (X_k^\nu, \xi_k) - \nabla f^\nu (x_k) \right\|$ is bounded.
\begin{lemma} \label{lemma:stochastic_gradient_2}
    For the stochastic gradient $g_{N_\nu}^\nu (X_k^\nu, \xi_k)$ in equation \eqref{eq:sto_grad_subproblem}, by Assumption \ref{assump:stochastic_gradient}, we have\\
    (a) $\EE\left[ g_{N_\nu}^\nu (X_k^\nu, \xi_k) \right] = \nabla f^\nu (x_k)$.\\
    (b) $\EE \left[ \left\| g_{N_\nu}^\nu (X_k^\nu, \xi_k) - \nabla f^\nu (x_k) \right\|^2 \right] \leq \sigma^2 + \frac{\sigma^2 + M}{N_\nu}$, where $M = \max_x \|\nabla f(x)\|^2$.
\end{lemma}
The proof is deferred to the following subsection. 

With the above propositions, the stochastic gradient is well-defined for the subproblems.
Next, we show that the classical convergence result of the SGD method can be directly achieved for subproblem $(\mathrm{P})$.

Suppose the objective function $f^\nu$ satisfies the P\L\ inequality, i.e.
\begin{align}
    \frac{1}{2} \|\nabla f^\nu (x)\|^2 \geq \mu \left( f^\nu(x) - f^\nu_*\right), \ \forall x \in \Omega.
\end{align}

The following theorem is derived directly from Theorem 4 in \cite{karimi2016linear}:
\begin{theorem}\label{thm:sgd_subprob}
    For the $\nu$th subproblem $(\mathrm{P}^\nu)$, suppose $f^\nu$ satisfies the P\L\ inequality, Assumption \ref{assump:Lipschitz} and Assumption \ref{assump:stochastic_gradient} are satisfied.
    For the Algorithm \ref{algorithm:SGD_subproblem}, \\
    (a) Let the stepsize $\alpha_k = \frac{2k + 1}{2\mu (k+1)^2}$ , then the convergence rate is
    \begin{align}
        \EE \left[ f^\nu(x_k) - f^\nu_*\right] \leq \frac{L}{2k\mu^2} (\sigma^2 + M) \left(1 + \frac{1}{N_\nu}\right),
    \end{align}
    where $M = \max_x \|\nabla f(x)\|^2$.\\
    (b) For constant stepsize $\alpha_k = \alpha < \frac{1}{2\mu}$, the convergence rate is
    \begin{align}
        \EE \left[ f^\nu(x_k) - f^\nu_*\right] \leq (1-2\mu\alpha)^k (f^\nu(x_0) - f^\nu_*) + \frac{L\alpha}{4\mu} (\sigma^2 + M) \left(1 + \frac{1}{N_\nu}\right).
    \end{align}
\end{theorem}

\subsection{Deferred proofs}
Proof of Lemma \ref{lemma:Lipschitz_subproblem}.
\begin{proof}
    For any $x,y \in \RR^d$,
    \begin{align*}
        \|\nabla f^\nu(x) - \nabla f^\nu(y)\| &= \left\| \int_{\RR^d} \phi_{1/\nu} (z) \nabla f(x - z) - \phi_{1/\nu} (z) \nabla f(y-z) \dd z \right\| \\
        &= \left\| \int_{\RR^d} \left(\nabla f(x-z) - \nabla f(y-z)\right) \phi_{1/\nu}(z) \ \dd z\right\| \quad \text{(Bochner integral)} \\
        &\leq  \int_{\RR^d} \left\| \left( \nabla f(x-z) - \nabla f(y-z)\right) \phi_{1/\nu} (z) \right\| \ \dd z \\
        &= \int_{\RR^d} \left\| \left( \nabla f(x-z) - \nabla f(y-z)\right)\right\| \phi_{1/\nu}(z) \ \dd z \\
        &\leq L \|x-y\| \int_{\RR^d} \phi_{1/\nu}(z)\ \dd z = L \|x-y\|.
    \end{align*}
\end{proof}

Proof of Lemma \ref{lemma:stochastic_gradient_1}.
\begin{proof}
    (a):
    \begin{align*}
        \EE \left[  g^\nu (x_k, \xi_k)\right] &= \EE \left[ \int_{\RR^d}  g(y,\xi_k) \phi_{1/\nu} (x_k - y) \ \dd y\right] \\
        &= \int_{\RR^d} \EE_{\xi_k} \left[  g(y,\xi_k) \right] \phi_{1/\nu} (x_k - y) \ \dd y \\
        &= \int_{\RR^d} \nabla f(y) \phi_{1/\nu} (x_k - y) \ \dd y  = \nabla f^\nu(x_k).
    \end{align*}
    (b):
    \begin{align*}
        &\EE \left[ \left\| g^\nu(x_k, \xi_k) - \nabla f^\nu (x_k) \right\|^2 \right] \\
        &= \EE \left[ \left\| \int_{\RR^d} g(y,\xi_k) \phi_{1/\nu} (x_k-y) \ \dd y - \int_{\RR^d} \nabla f(y) \phi_{1/\nu} (x_k-y) \ \dd y \right\|^2 \right] \\
        &= \EE \left[ \left\| \int_{\RR^d} \left(g(y,\xi_k) - \nabla f(y)  \right) \phi_{1/\nu} (x_k-y) \ \dd y\right\|^2 \right] \\
        &\leq \EE \left[  \left( \int_{\RR^d} \left\|\left(g(y,\xi_k) - \nabla f(y)  \right) \phi_{1/\nu} (x_k-y)\right\| \ \dd y \right)^2 \right] \ \text{Bochner integral}\\
        &= \EE \left[  \left( \int_{\RR^d} \left\|g(y,\xi_k) - \nabla f(y) \right\| \phi_{1/\nu} (x_k-y) \ \dd y \right)^2 \right]\\
        &\leq \EE \left[\left( \max_y\left( \left\|g(y,\xi_k) - \nabla f(y)\right\| \right)  \int_{\RR^d}  \phi_{1/\nu} (x_k-y) \ \dd y \right)^2 \right] \\
        &= \EE \left[\left( \max_y\left( \left\|g(y,\xi_k) - \nabla f(y)\right\| \right)\right)^2 \right] \leq \sigma^2.
    \end{align*}
\end{proof}

Proof of Lemma \ref{lemma:stochastic_gradient_2}.
\begin{proof}
    (a): 
    \begin{align*}
        \EE\left[ g_{N_\nu}^\nu (X_k^\nu, \xi_k) \right] &= \EE \left[ \EE \left[ g_{N_\nu}^\nu (X_k^\nu, \xi_k) | \xi_k\right]\right] \\
        &= \EE \left[ \EE \left[ \frac{1}{N_\nu} \sum_{i=1}^{N_\nu} g ((X_k^\nu)_i, \xi_k) | \xi_k\right]\right] \\
        &= \EE \left[ \EE \left[ g (X_k^\nu, \xi_k) | \xi_k\right]\right] \\
        &= \EE \left[ \int_{\RR^d} g(y,\xi_k) \phi_{1/\nu}(x_k - y) \ \dd y \right] \\
        &= \EE \left[ g^\nu (x_k,\xi_k)\right] = \nabla f^\nu(x_k) \quad \text{by Lemma \ref{lemma:stochastic_gradient_1} (a)}.
    \end{align*}
    (b):
    First,
    \begin{align*}
        \EE \left[  g_{N_\nu}^\nu (X_k^\nu, \xi_k) | \xi_k \right] &= \EE \left[  \left. \frac{1}{N_\nu} \sum_{i=1}^{N_\nu} g ((X_k^\nu)_i, \xi_k) \right| \xi_k \right]\\
        &= \EE \left[ \left. g(X_k^\nu, \xi_k) \right| \xi_k \right] = \int_{\RR^d} g(y,\xi_k) \phi_{1/\nu}(x_k - y) \ \dd y = g^\nu (x_k, \xi_k).
    \end{align*}
    Denote $\epsilon(\xi_k) = g^\nu (x_k, \xi_k) - \nabla f^\nu (x_k)$. By Lemma \ref{lemma:stochastic_gradient_1} (a) and (b),
    \begin{align}
        \EE[\epsilon(\xi_k)] = 0, \quad \EE[\|\epsilon(\xi_k)\|^2] \leq \sigma^2.
    \end{align}
    Then, we have the bias–variance decomposition:
    \begin{align*}
        &\EE \left[ \left\| g_{N_\nu}^\nu (X_k^\nu, \xi_k) - \nabla f^\nu (x_k) \right\|^2 \right] \\
        &= \EE \left[ \left\| g_{N_\nu}^\nu (X_k^\nu, \xi_k) - g^\nu(x_k,\xi_k) + g^\nu(x_k,\xi_k) - \nabla f^\nu (x_k) \right\|^2 \right]\\
        &= \EE \left[ \left\| g_{N_\nu}^\nu (X_k^\nu, \xi_k) - g^\nu(x_k,\xi_k) + \epsilon(\xi_k) \right\|^2 \right] \\
        &= \EE \left[ \left\| g_{N_\nu}^\nu (X_k^\nu, \xi_k) - g^\nu(x_k,\xi_k) \right\|^2 \right] + 2\EE\left[ \epsilon(\xi_k)^T \left( g_{N_\nu}^\nu (X_k^\nu, \xi_k) - g^\nu(x_k,\xi_k) \right)\right] + \EE\left[ \|\epsilon(\xi_k)\|^2 \right] \\
        &= \EE \left[ \left\| g_{N_\nu}^\nu (X_k^\nu, \xi_k) - g^\nu(x_k,\xi_k) \right\|^2 \right] + \EE\left[ \|\epsilon(\xi_k)\|^2 \right].
    \end{align*}
    Next,
    \begin{align*}
        &\EE \left[ \left\| g_{N_\nu}^\nu (X_k^\nu, \xi_k) - g^\nu(x_k,\xi_k) \right\|^2 \right]\\
        &= \EE \left[ \EE \left[ \left. \sum_{l=1}^d \left( g_{N_\nu}^\nu (X_k^\nu, \xi_k)_l - g^\nu (x_k, \xi_k)_l \right)^2\right| \xi_k \right]\right] \\
        &= \EE \left[\sum_{l=1}^d \VV\left[ \left. g_{N_\nu}^\nu (X_k^\nu, \xi_k)_l \right| \xi_k \right] \right] = \EE \left[\sum_{l=1}^d \VV\left[ \left. \frac{1}{N_\nu} \sum_{i=1}^{N_\nu} g ((X_k^\nu)_i, \xi_k)_l \right| \xi_k \right] \right] \\
        &= \frac{1}{N_\nu} \EE \left[  \sum_{l=1}^d \VV\left[ \left.g (X_k^\nu, \xi_k)_l \right| \xi_k \right] \right] \\
        &\leq \frac{1}{N_\nu} \EE \left[ \sum_{l=1}^d \EE\left[ \left. \left(g (X_k^\nu, \xi_k)_l\right)^2 \right| \xi_k \right] \right] = \frac{1}{N_\nu} \EE \left[\EE\left[ \left. \left\|g (X_k^\nu, \xi_k)\right\|^2 \right| \xi_k \right] \right] \\
        &= \frac{1}{N_\nu} \EE \left[ \int_{\RR^d} \left\|g (y, \xi_k)\right\|^2 \phi_{1/\nu}(x_k - y) \ \dd y \right] \\
        &= \frac{1}{N_\nu}  \int_{\RR^d} \EE \left[ \left\|g (y, \xi_k)\right\|^2\right] \phi_{1/\nu}(x_k - y) \ \dd y \leq \frac{\sigma^2 + M}{N_\nu},
    \end{align*}
    where $M = \max_x \|\nabla f(x)\|^2$.
    Then,
    \begin{align*}
        \EE \left[ \left\| g_{N_\nu}^\nu (X_k^\nu, \xi_k) - \nabla f^\nu (x_k) \right\|^2 \right] \leq \sigma^2 + \frac{\sigma^2 + M}{N_\nu}.
    \end{align*}
\end{proof}

Proof of Theorem \ref{thm:sgd_subprob}.
\begin{proof}
    By the construction of $f^\nu$ and Proposition \ref{prop:8.10}, $f^\nu$ is continuous on the closed set $\Omega$.
    Then subproblem ($P^\nu$) has a non-empty solution set.

    By Lemma \ref{lemma:stochastic_gradient_2} (a) and (b),
    \begin{align*}
        \EE \left[ \left\| g_{N_\nu}^\nu (X_k^\nu, \xi_k) - \nabla f^\nu (x_k) \right\|^2 \right] 
        &= \EE \left[ \left\| g_{N_\nu}^\nu (X_k^\nu, \xi_k)\right\|^2 - 2\left< g_{N_\nu}^\nu (X_k^\nu, \xi_k), \nabla f^\nu (x_k) \right> + \|\nabla f^\nu (x_k)\|^2  \right]\\
        &\leq \EE \left[ \left\| g_{N_\nu}^\nu (X_k^\nu, \xi_k)\right\|^2 \right] - \|\nabla f^\nu (x_k)\|^2
        \leq \sigma^2 + \frac{\sigma^2 + M}{N_\nu}.
    \end{align*}
    Also,
    \begin{align*}
        \|\nabla f^\nu (x_k)\|^2 &= \left\| \int_{\RR^n} \nabla f (x_k-y) \phi_{1/\nu}(y) \ \dd y  \right\|^2 \\
        &\leq \left( \int_{\RR^n} \left\| \nabla f (x_k-y) \phi_{1/\nu}(y) \right\| \ \dd y \right)^2 \\
        &\leq \left( \max_y (\left\| \nabla f (x_k-y) \right\| ) \int_{\RR^n} \phi_{1/\nu}(y)  \ \dd y \right)^2 = \max_x \|\nabla f(x)\|^2 \eqdef M.
    \end{align*}
    Then,
    \begin{align*}
        \EE \left[ \left\| g_{N_\nu}^\nu (X_k^\nu, \xi_k)\right\|^2 \right] \leq \sigma^2 + \frac{\sigma^2 + M}{N_\nu} + M.
    \end{align*}
    The proof is finished directly by Theorem 4 in \cite{karimi2016linear}.
\end{proof}

\section{Synthetic numerical experiments}
In this section, we provide two synthetic toy experiments to demonstrate that the ML-SGD scheme has the potential to overcome the local minimum solution in a low-dimensional setting.

\begin{figure}[h]
    \centering
    \includegraphics[width=1\textwidth]{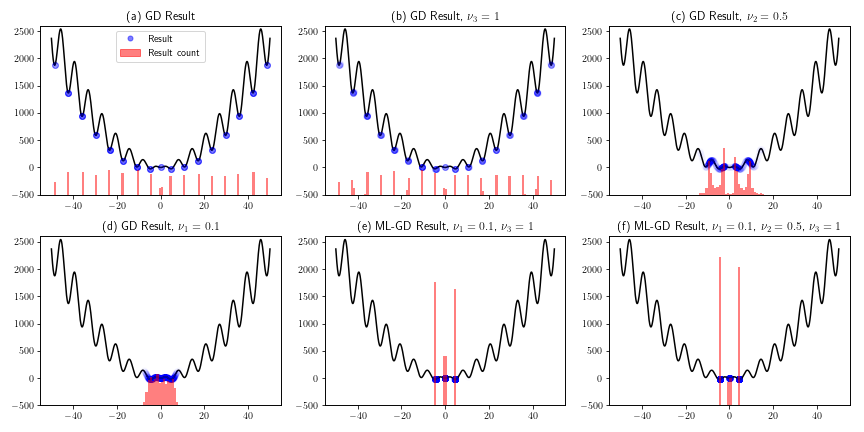}
    \caption{Distribution of the finial results. (a): The gradient descent (GD) algorithm is performed on the objective function $f_1(x)$. (b), (c), (d): Algorithm \ref{algorithm:SGD_subproblem} with accurate gradients instead of stochastic gradients is performed on $f_1(x)$ with different smooth coefficients. (e), (f): two ML-GD schemes are performed on $f_1(x)$ with different smoothing strategy.}
    \label{fig:toy_prob_1}
\end{figure}

The first example is
\begin{align}
    \min_x f_1(x), \quad \text{where }f_1(x) = x^2 + 10 x \sin(x).
\end{align}
This example is also used in \cite{Harshvardhan2021Escaping}.
The objective function $f_1$ is highly nonlinear, with two global minima located near $\pm 4.7$ \cite{Harshvardhan2021Escaping}.
The gradient descent algorithm is used to optimize the actual objective function $f_1$.
Algorithm \ref{algorithm:SGD_subproblem} is used to optimize the smoothed objective functions $f_1^{\nu_i}$, where $\nu_1 = 0.1$, $\nu_2 = 0.5$, $\nu_3 = 1$, and $N_\nu = 10$.
The standard Gaussian kernel is used.
Notice that $\nu_1$ and $\nu_2$ are not integers in this example.
This will not affect the theoretical properties of Algorithm \ref{algorithm:SGD_subproblem} and \ref{algorithm:ML-SGD} since we can construct a new ``base'' kernel $\phi^{'} = \phi_{1/\nu_1}$, then new smooth coefficients can be achieved as $\nu_1^{'} = 1$, $\nu_2^{'} = 5$, and $\nu_3^{'} = 10$.
For comparison, we fix the step size with $\alpha = 0.001$, and for each optimization process, the total iteration number is $1000$.
We generate 1000 random initial values in $[-50,50]$ with uniform distribution.
First, the gradient descent algorithm is performed to minimize $f_1$ with the above randomly generated initial values.
The result is shown in Figure \ref{fig:toy_prob_1} (a).
Next, The minimization result of smoothed objective functions $f^{\nu_1}$, $f^{\nu_2}$, and $f^{\nu_3}$ are shown in Figure \ref{fig:toy_prob_1} (b), (c), and (d).
Then, the multi-layer strategy is performed with two layers ($\nu_1$ and $\nu_3$) and three layers ($\nu_1$, $\nu_2$, and $\nu_3$) by Algorithm \ref{algorithm:ML-SGD}.
Since no stochastic gradient is used, we denote this case as ML-GD.
The results are shown in Figure \ref{fig:toy_prob_1} (e) and (f).

As can be seen from subfigure (a), when the actual objective function is solved directly, most of the results are trapped in the local minimum.
For smoothed problems, different degrees of smoothness affect the results of optimization.
When $\nu$ is large, the objective function $f^\nu$ is less smooth.
Thus, the results in subfigure (b) are more trapped in the local minima than (c) and (d), which are farther away from the global minima.
Subfigures (e) and (f) show evidence that the multi-layer structure can indeed improve the optimization result.
With the multi-layer strategy, the optimization results are largely improved.
Compared to other cases, there are more optimization results that fall into the global minimum in (e) and (f).
In addition, by designing a multi-layer structure that transitions more smoothly, the optimization results can be further improved.

\begin{figure}[h]
    \centering
    \includegraphics[width=0.7\textwidth]{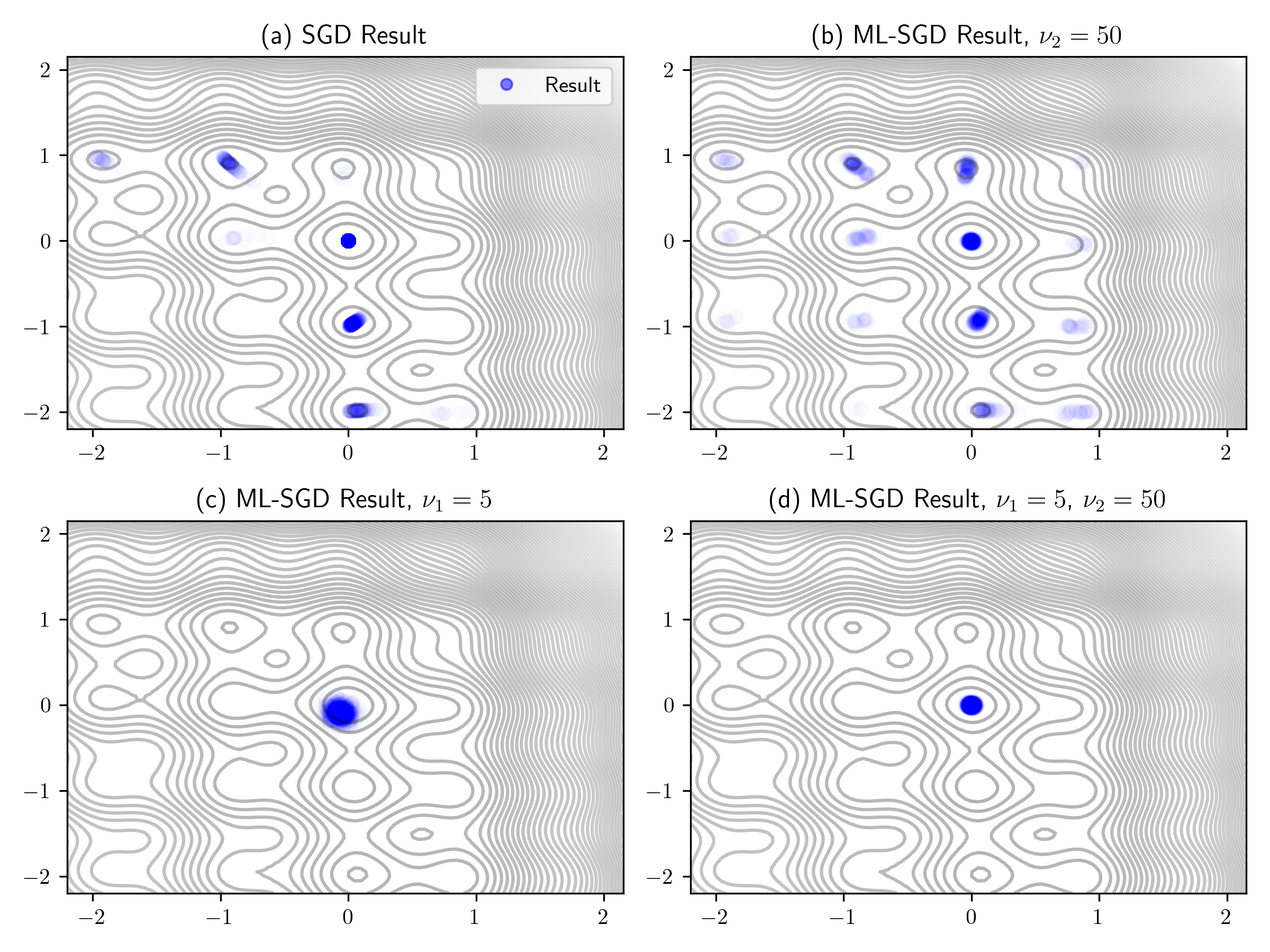}
    \caption{Distribution of the finial results. (a): The SGD method is performed to minimize the objective function $f_2(x)$. (b), (c): Algorithm \ref{algorithm:SGD_subproblem} are performed on $f_2(x)$ with different smooth coefficients. (d): A ML-SGD result on $f_2(x)$. }
    \label{fig:toy_prob_2}
\end{figure}

The second example is a two-dimensional toy problem
\begin{align}
    \min_x \ (f_2 (x) &\defeq f_R(x) + f_H(x)) \\
    \text{where } f_R(x) &= 20 + (x_1^2 - 10\cos(2\pi x_1)) + (x_2^2 - 10\cos(2\pi x_2)) \\
    f_H(x) &=  ((x_1+3)^2 + (x_2+2) - 11)^2 + ((x_1+3) + (x_2+2)^2 - 7)^2.
\end{align}
The global minimum is $(0,0)$.
We simulate the stochastic gradient by letting $g(x_k) = \nabla f_R(x_k)$ or $g(x_k) =\nabla f_H(x_k)$ with equal probability at $k$th iteration.
There are 1000 random initial values generated with uniform distribution in $[-2,2]\times [-2,2]$.

The SGD method is used to minimize the actual objective function $f_2(x)$, and the result is shown in Figure \ref{fig:toy_prob_2} (a).
Although many local minimum solutions are shown in Figure \ref{fig:toy_prob_2} (a), there are still many solutions around the global minimum.
A recent work \cite{kleinberg2018alternative} suggests that the traditional SGD algorithm can also be explained by the graduated optimization approach, i.e., the random properties of gradients are equivalent to smoothing the objective function.
Let $\nu_1 = 5$, $\nu_2 = 50$, $N_\nu = 10$.
Then, Algorithm \ref{algorithm:SGD_subproblem} is used to minimize the smoothed objective function $f_2^{\nu_1}(x)$ and $f_2^{\nu_2}(x)$ with Gaussian kernel.
As shown in subfigures (b) and (c), different degrees of smoothness affect the optimization results.
When using Algorithm \ref{algorithm:SGD_subproblem} to minimize $f_2^{\nu_1}$, most of the results are located near the global minimum.
Based on this result, we minimize $f_2$ by Algorithm \ref{algorithm:ML-SGD} with $\nu_1 = 5$ and $\nu_2 = 50$ as the first and the second layer.
The results are shown in subfigure (d), where most of the results are very closed to the global minimum.

\section{Discussion and outlook}

In this work, the SGD method is studied with the graduated optimization approach.
Under this view of point, the ML-SGD scheme is proposed, and the convergence of graduated optimization is studied.
Notice that our main theorem \ref{thm:main_result1} only provides an asymptotic behavior of graduated optimization, i.e., given the actual problem $(\mathrm{P})$, we can construct a series of smoothed subproblems $(\mathrm{P}^\nu)$ such that the solutions of the subproblems converge to the solution of the actual problem.
Thus, it is still a heuristic decision to choose the kernel $\phi$ and the smooth parameter $\nu$ for a given optimization problem.
It is also worth noting that although ML-SGD performs well in low-dimensional situations, the effect on high-dimensional problems still needs to be experimentally verified.
Especially for the deep neural network training problems, evidence shows that the transition in the loss landscapes between the chaotic area and nearly convex area is quick \cite{li2017visualizing}.
As a result, once the initial value or the model in the current iteration is in a nearly convex area, the role of graduate optimization may not be obvious.
This makes the application value of graduate optimization and ML-SGD scheme need further research.

On the other hand, the algorithmic structure of graduated optimization through random sampling is particularly suitable for the distributed computing environment of contemporary machine learning tasks.
With only minor changes, the graduated optimization approach can be deployed in massively distributed computing problems, which opens up the possibility for further study of the practical role of graduated optimization.
In addition, as we built the ML-SGD algorithm, the idea of multi-layer structure can be used for other first-order optimization algorithms, such as Adagrad, RMSprop, Adam, etc.
Continued research in this direction could lead to more effective training algorithms for large-scale neural network models.

\section{Acknowledgment}

The work of the first and second authors was funded by NSERC (Natural Sciences and Engineering Research Council of Canada) through the grant NSERC Discovery Grant (RGPIN-2018-04328).
The work of the first and third authors was funded by NSERC Discovery Grant (RGPIN-2018-05147), NSERC RTI Grant (RTI-2021-00675), University of Calgary VPR Catalyst grant, New Frontiers in Research Fund (NFRFE-2018-00748).
The authors would like to thank the support.

\newpage
\bibliographystyle{amsplain}
\bibliography{references}
\end{document}